\documentclass[10pt]{amsart}
\usepackage[utf8x]{inputenc}

\usepackage{amsmath}
\usepackage{amssymb}
\usepackage{graphicx}
\usepackage{amsfonts}
\usepackage{mathpazo}
\usepackage{color}
\usepackage{paralist}
\usepackage{enumitem}
\usepackage{amsxtra}
\usepackage{latexsym,mathrsfs}
\usepackage[hidelinks]{hyperref}
\usepackage[normalem]{ulem}
\newtheorem{theorem}{Theorem}[section]

\newtheorem{lemma}{Lemma}[section]
\newtheorem{proposition}{Proposition}[section]
\newtheorem{corollary}{Corollary}[section]
\newcommand{\Nr}{N^{r}}
\newcommand{\Nb}{N^{b}}

\newcommand{\db}{d^{b}}
\newcommand{\ceil}[1]{\left\lceil #1\right\rceil}
\newcommand{\floor}[1]{\left\lfloor #1\right\rfloor}

\title{Sharp Same-Color Cycle Covers in Two-Colored Complete Graphs}
\author{Xiao-Chuan Liu}
\address[Liu]{Departamento de Matemática,
 Universidade Federal de Pernambuco,
	Avenida Jornalista Aníbal Fernandes, Cidade Universitária, Recife, Brazil}
\email{xiaochuan.liu@ufpe.br}

\author{Jonatas Teodomiro}
\address[Teodomiro]{Departamento de Matemática,
 Universidade Federal de Pernambuco,
	Avenida Jornalista Aníbal Fernandes, Cidade Universitária, Recife, Brazil}
\email{jonatas.teodomiro@ufpe.br}

\author{Xu Yang}
\address[Yang]{Instituto de Computação,  Universidade Federal de Alagoas,
	Av. Lourival Melo Mota, S/N, Maceió, Brazil}
\email{yang@ic.ufal.br}

\begin{document}
\maketitle

\begin{abstract}
We extend the conjecture of Erd\H{o}s and Gy\'arf\'as on monochromatic path covers to the setting of monochromatic cycle covers. We prove that, for all $n$, every 2-edge-coloring of the complete graph on $n$ vertices contains a collection of at most $\lceil\sqrt{n}\rceil$ monochromatic cycles, all of the same color, that together cover all vertices.  The order of
the bound is best possible, and the ceiling is necessary for infinitely
many $n$.   
\end{abstract}

\section{Introduction}

In 1967, Gerencs\'er and Gy\'arf\'as~\cite{gerencser1967ramsey} proved
that the vertex set of every 2-edge-colored complete graph can be
partitioned into a red path and a blue path.  In 1979, Lehel conjectured
that the same statement remains true when paths are replaced by cycles.
Although replacing paths by cycles appears to be a natural extension,
the resulting problem is substantially more difficult.  Indeed, before
the conjecture was completely resolved, it was proved only for
sufficiently large complete graphs by \L uczak, R\"odl, and
Szemer\'edi~\cite{luczak1998partitioning} in 1998 using Szemer\'edi's
Regularity Lemma, and later Allen~\cite{allen2008covering} provided a
simpler proof in 2008.  Finally, Bessy and
Thomass\'e~\cite{bessy2010partitioning} settled Lehel's conjecture in
full in 2010 by giving a complete proof that does not rely on the
Regularity Lemma.

\begin{theorem}[Bessy and Thomass\'e~\cite{bessy2010partitioning}]
\label{thm:cycle-partition}
The vertex set of any 2-edge-colored complete graph \(K_n\) can be
partitioned into a red cycle and a blue cycle.
\end{theorem}

The problem of covering vertices of edge-colored complete graphs by
monochromatic paths and cycles has also been studied more generally.
Gyárfás~\cite{gyarfas1983vertex} investigated vertex coverings by
monochromatic paths and cycles, while Erdős, Gyárfás, and
Pyber~\cite{erdosgyarfaspyber1991vertex} initiated a systematic study
of coverings and partitions of multicolored complete graphs by
monochromatic cycles and trees.  In these results the covering pieces
are allowed to have different colors.  The problem considered here is
more restrictive: all cycles in the cover are required to have the
same color.

Motivated by these partition problems, one may more generally consider
covering problems in edge-colored complete graphs.  In this paper, we
study coverings of edge-colored complete graphs by monochromatic
structures.  A well-known conjecture of Erd\H{o}s and Gy\'arf\'as
concerns covering the vertex set of a 2-edge-colored complete graph by
a small number of monochromatic paths.  As noted in Gy\'arf\'as'
survey~\cite{gyarfas2016vertex}, when the problem was first discussed
with Erd\H{o}s, a misunderstanding led them to consider coverings by
monochromatic paths all in the same color.  This perspective initiated
the study of minimizing the number of same-color monochromatic paths
covering all vertices of a 2-edge-colored complete graph.  Erd\H{o}s
and Gy\'arf\'as~\cite{erdos1995vertex,gyarfas2016vertex} conjectured
that the vertex set of every 2-edge-colored complete graph can be
covered by at most \(\sqrt n\) monochromatic paths all in the same
color.  Pokrovskiy, Versteegen, and
Williams~\cite{pokrovskiy2026proof} proved the conjecture for all
sufficiently large \(n\).  Very recently, Chen and
Chen~\cite{chen2026allorders} completed the problem by proving it for
every positive integer \(n\).

Eugster and Mousset~\cite{eugster2018vertex} studied a more general
version of this covering problem in which the edges of $K_n$ are
colored with \(r\) colors and the covering pieces are allowed to use
at most \(s\) colors in total. Their results apply both to paths and
to cycles. In particular, in the two-color setting with only one
color allowed in the cover, their Corollary~3 implies that the minimum
number of same-color monochromatic cycles needed in the worst case is
of order $\Theta(\sqrt n).$
Thus the correct order of magnitude for the cycle problem was already
known, but the multiplicative constant, and hence the sharp bound,
remained undetermined.

The present paper determines this bound sharply. We prove that for
every positive integer \(n\), every red--blue edge-coloring of \(K_n\)
admits a cover by at most \(\lceil\sqrt n\rceil\) monochromatic cycles,
all of the same color. Moreover, our construction below shows that
this bound is attained for infinitely many \(n\). In particular, our
result strengthens the asymptotic \(\Theta(\sqrt n)\) estimate of
Eugster and Mousset to the sharp universal bound $\lceil \sqrt{n}\rceil$, attained for infinitely many $n$. As throughout the paper, a single vertex and a
single edge are regarded as degenerate cycles, and the cycles in a
cover are allowed to intersect.
  Our main result is stated in the following theorem.

\begin{theorem}\label{thm:main}
For every positive integer \(n\), every red--blue edge-coloring of
\(K_n\) has a cover by at most \(\ceil{\sqrt n}\) monochromatic cycles,
all of one color.
\end{theorem}

Note that when \(\sqrt n\) is not an integer, Pokrovskiy, Versteegen, and
Williams~\cite{pokrovskiy2026proof} and the all-order path theorem
of Chen and Chen~\cite{chen2026allorders} give a cover by
\(\lfloor\sqrt n\rfloor\) monochromatic paths of the same color.  In
contrast, Theorem~\ref{thm:main} gives
\(\lceil\sqrt n\rceil\) monochromatic cycles of the same color.  Thus,
when \(\sqrt n\) is not an integer, the cycle version may require one
more monochromatic component than the path version.  The following
example shows that this gap of one is unavoidable.

Let \(n\) be such that \(\sqrt{n+1}\) is an integer.  Partition
\(V(K_n)=A\cup B\), where \(|A|=n-\floor{\sqrt n}\) and
\(|B|=\floor{\sqrt n}\), and color all edges inside \(A\) blue and
all remaining edges red.
Every blue cycle is contained in \(A\), except for degenerate
single-vertex cycles in \(B\), so covering all vertices in blue
requires \(\floor{\sqrt n}+1=\ceil{\sqrt n}\) cycles.  On the other
hand, every red cycle contains at most \(\floor{\sqrt n}\) vertices of
\(A\).  Therefore, covering \(A\) with red cycles requires at least
\[
 \left\lceil\frac{n-\floor{\sqrt n}}{\floor{\sqrt n}}\right\rceil
 =\left\lceil\frac{n}{\floor{\sqrt n}}-1\right\rceil
 =\sqrt{n+1}=\ceil{\sqrt n}
\]
red cycles.  This shows that the bound \(\ceil{\sqrt n}\) is best
possible.  In particular, when \(\sqrt n\) is not an integer, the
difference of one between the path and cycle bounds is necessary for infinitely many $n$. 

We conclude with a brief overview of the proof.  We argue by means of a
minimum counterexample and choose a longest monochromatic cycle \(C\),
say blue.  Writing \(X=V(C)\) and \(Y=V(K_n)\setminus X\), minimality
gives a sharp restriction on the intersection of red and blue cycles.
The successors on \(C\) of the blue neighbors of a vertex in \(Y\)
form a red clique; this successor structure is the main ingredient in
our extension arguments.  Bipartite cycle lemmas are then used either
to construct a red cycle spanning \(Y\), after which an exact packing
argument completes the cover, or to enlarge the successor structure
until the intersection restriction is contradicted.  This gives a
uniform proof when \(\ceil{\sqrt n}\ge5\).  The remaining cases require
additional small structural configurations, which are treated in the
appendices.

\section{Bipartite tools}

For a bipartite graph \(H[A,B]\) and \(T\subseteq V(H)\), define
\(\sigma(T)\) to be the minimum of \(d_H(a)+d_H(b)\) over all
nonedges \(ab\), where \(a\in A\), \(b\in B\), and
\(\{a,b\}\cap T\ne\varnothing\), with the minimum understood as
infinity when there is no such nonedge.  We shall repeatedly use the
following cyclability theorem.

\begin{theorem}[Okamura and Yamashita~\cite{okamura2013degree}, Theorem~5]
\label{thm:OY}
Let \(H[A,B]\) be a \(2\)-connected bipartite graph with
\(|A|\ge |B|\), and let \(T\subseteq V(H)\).  If
\(\sigma(T)\ge |A|+1\), then either \(H\) has a cycle containing
every vertex of \(T\), or \(|T\cap A|>|B|\) and \(H\) has a cycle
containing every vertex of \(B\).
\end{theorem}

The next lemma is the cycle counterpart of the bipartite path lemma of
Pokrovskiy, Versteegen, and Williams~\cite[Lemma~2.2]{pokrovskiy2026proof}.
Their common-neighbor argument closes cyclically under the present
degree condition, so we include the short proof.

\begin{lemma}\label{lem:smallclass}
Let \(H[X,Y]\) be bipartite, where \(|X|\ge |Y|=m\ge2\).  If
\(d_H(y)\ge (|X|+m)/2\) for every \(y\in Y\), then \(H\) has a cycle
containing every vertex of \(Y\) and exactly
\(m\) vertices of \(X\).
\end{lemma}

\begin{proof}
Order \(Y=\{y_1,\ldots,y_m\}\) cyclically and put
\(A_i=N_H(y_i)\cap N_H(y_{i+1})\), where \(y_{m+1}=y_1\).
Then \(|A_i|\ge m\) for every \(i\).  Since there are \(m\) such sets,
we can successively choose distinct \(x_i\in A_i\).  The cycle
\(y_1x_1y_2x_2\cdots y_mx_my_1\) has the required form.
\end{proof}

The next two Hamilton-cycle lemmas are direct consequences of
Theorem~\ref{thm:OY}; their proofs verify its 2-connectivity hypothesis.

\begin{lemma}\label{lem:nearcomplete}
Let \(H[A,B]\) be bipartite with \(|A|=|B|=s\ge2\).  If
\(d_H(a)\ge2\) for \(a\in A\) and \(d_H(b)\ge s-1\) for \(b\in B\),
then \(H\) has a Hamilton cycle.
\end{lemma}

\begin{proof}
For \(s=2\), every vertex of \(A\) is complete to \(B\).
Let \(s\ge3\).  We first check that \(H\) is \(2\)-connected.
After deleting a vertex of \(A\), any two vertices of \(B\) have a
common neighbor if \(s\ge4\), since \(2(s-2)-(s-1)=s-3\ge1\).
When \(s=3\), if the two remaining \(A\)-vertices were in different
components, their two disjoint neighborhoods in the three-set \(B\)
would both have size at least two, which is impossible.  Every
remaining vertex of \(B\) still has a neighbor.  After deleting a
vertex of \(B\), every vertex of \(A\) retains a neighbor, and every
two remaining vertices of \(B\) have a common neighbor in \(A\).
Thus deletion of any vertex leaves a connected graph.

For every nonedge \(ab\), with \(a\in A\) and \(b\in B\), we have
\(d_H(a)+d_H(b)\ge2+(s-1)=s+1\).
Apply Theorem~\ref{thm:OY} with \(T=A\).  Since
\(|T\cap A|=|B|=s\), the exceptional alternative is impossible, so a
cycle contains \(A\).  The two classes are balanced, hence this cycle
is Hamiltonian.
\end{proof}

\begin{lemma}\label{lem:degreesum}
Let \(s\) and \(\lambda\) be nonnegative integers with
\(s\ge 2\lambda+2\), and let \(H[A,B]\) be bipartite with \(|A|=|B|=s\).  Suppose
\(d_H(a)\ge\lambda+1\) for \(a\in A\), and
\(d_H(b)\ge s-\lambda\) for \(b\in B\).  Then \(H\) has a Hamilton
cycle.
\end{lemma}

\begin{proof}
After deleting a vertex of \(A\), the vertices of \(B\) retain at least
\(2(s-\lambda-1)-(s-1)=s-2\lambda-1\ge1\) common neighbors.  After
deleting a vertex of \(B\), every vertex of
\(A\) retains a neighbor and the remaining vertices of \(B\) have at
least \(s-2\lambda\ge2\) common neighbors.  Hence \(H\) is
\(2\)-connected.  Every nonedge \(ab\), \(a\in A,b\in B\), satisfies
\(d_H(a)+d_H(b)\ge(\lambda+1)+(s-\lambda)=s+1\).
Theorem~\ref{thm:OY}, again with \(T=A\), gives a cycle
containing \(A\), which is Hamiltonian.
\end{proof}

\begin{lemma}\label{lem:defectselection}
Let \(H[A,B]\) be bipartite with \(|B|=w\).  Suppose one vertex of
\(B\) is complete to \(A\), and every vertex of \(B\) has at most
\(\lambda\) non-neighbors in \(A\).  If
\[
 w\ge2\lambda+2\qquad and \qquad
 |A|-w\ge\floor{\frac{\lambda(w-1)}{w-\lambda}},
\]
then \(H\) has a cycle containing \(B\) and exactly \(w\) vertices of
\(A\).
\end{lemma}

\begin{proof}
Let \(L=\{a\in A:d(a,B)\le\lambda\}\).  There are at most
\(\lambda(w-1)\) nonedges, while every member of \(L\) is incident
with at least \(w-\lambda\) of them.  Hence
\(|L|\le\floor{\lambda(w-1)/(w-\lambda)}\).
Choose a \(w\)-set \(A'\subseteq A\setminus L\).
Every vertex of \(A'\) has degree at least \(\lambda+1\), and every
vertex of \(B\) has degree at least \(w-\lambda\) into \(A'\).
Lemma~\ref{lem:degreesum} gives a Hamilton cycle on \(A'\cup B\).
\end{proof}

\begin{lemma}\label{lem:endpoints}
Let $m$ be a positive integer at least 2 and let \(H[A,B]\) be bipartite with \(|A|=m+1\), \(|B|=m\),
\(d_H(a)\ge2\) for \(a\in A\), and \(d_H(b)\ge m\) for \(b\in B\).
For every two distinct \(u,v\in A\), \(H\) has a Hamilton \(u\)--\(v\)
path.
\end{lemma}

\begin{proof}
Fix distinct \(u,v\in A\), and let \(H'=H-u\).  For every
\(a\in A\setminus\{u\}\) we have \(d_{H'}(a)\ge2\), while every
\(b\in B\) satisfies \(d_{H'}(b)\ge m-1\).  Hence
Lemma~\ref{lem:nearcomplete} gives a Hamilton cycle \(C\) of \(H'\).

Let \(b_1,b_2\in B\) be the two neighbors of \(v\) on \(C\).
If \(u\) is adjacent to \(b_i\) for some \(i\in\{1,2\}\), delete
the edge \(b_iv\) from \(C\) and prepend the edge \(ub_i\).
The resulting path is a Hamilton \(u\)--\(v\) path of \(H\).

We may therefore assume that \(u\) is adjacent to neither \(b_1\)
nor \(b_2\).  Since \(d_H(b_i)\ge m\) and \(|A|=m+1\), each
\(b_i\) has at most one non-neighbor in \(A\).  Thus both \(b_1\)
and \(b_2\) are adjacent to every vertex of \(A\setminus\{u\}\).

Choose \(y\in N_H(u)\).  Since \(ub_1,ub_2\notin E(H)\), we have
\(y\notin\{b_1,b_2\}\).  Let \(p,q\in A\setminus\{u\}\) be the two
neighbors of \(y\) on \(C\), labeled so that the \(b_1\)--\(b_2\)
arc of \(C\) avoiding \(v\) has the form
\[
 b_1\,P\,p\,y\,q\,Q\,b_2,
\]
where \(P\) and \(Q\) may be trivial.  Since \(b_1\) is adjacent
to every vertex of \(A\setminus\{u\}\), in particular \(b_1q\in
E(H)\).  Therefore
\[
 u\,y\,p\,P^{-1}\,b_1\,q\,Q\,b_2\,v
\]
is a Hamilton \(u\)--\(v\) path of \(H\).
\end{proof}

The packing idea in the next two lemmas is adapted from the bipartite
path-packing lemmas of Pokrovskiy, Versteegen, and
Williams~\cite[Lemmas~2.3 and~2.4]{pokrovskiy2026proof} and their
all-order refinement by Chen and Chen~\cite[Lemma~3.5]{chen2026allorders}.
Here the paths are closed into cycles by the two preceding Hamiltonian
tools.

\begin{lemma}\label{lem:packing}
Let \(D,Y_0,Y_1\) be the classes of a red bipartite graph, where
\(|Y_0|=a_0\ge2\).  Suppose that every \(u\in Y_0\) has at most one
non-neighbor in \(D\), and every \(z\in Y_1\) has at most \(\mu\)
non-neighbors in \(D\).  Put
\(E=\{x\in D:d(x,Y_0)\le1\}\), \(D^*=D\setminus E\), and \(d=|D^*|\).
Then \(|E|\le1\) if \(a_0\ge3\), and \(|E|\le2\) if \(a_0=2\).
Assume \(d\ge a_0\), let \(h\ge1\), and put \(p=d-ha_0\).
If \(p\le0\), then \(D^*\) is covered by at most \(h\) red cycles.
If \(p>0\), put \(q=\min\{p,h\}\).  The same conclusion holds provided
\begin{equation}\label{eq:P}
 p\le q|Y_1|,\qquad d-2\mu\ge p-q,\qquad d-\mu\ge p+q.
\end{equation}
Vertices of \(Y_0\cup Y_1\) may be reused by different cycles.
\end{lemma}

\begin{proof}
There are at most \(a_0\) nonedges between \(D\) and \(Y_0\).
Every member of \(E\) is incident with at least \(a_0-1\) of them,
which gives the stated bound on \(|E|\).  Every vertex of \(D^*\) has
at least two red neighbors in \(Y_0\).

If \(p\le0\), split \(D^*\) into at most \(h\) blocks of size at most
\(a_0\), padding the last block with previously used vertices.
Lemma~\ref{lem:nearcomplete}, applied to each padded block and \(Y_0\),
closes every block.

Suppose \(p>0\).  The first inequality in \eqref{eq:P} allows positive
integers \(b_1,\ldots,b_q\le|Y_1|\) with sum \(p\). Take an ordered \(b_i\)-tuple of distinct \(Y_1\)-vertices
(different tuples may overlap). To form $q$ red paths, use common red neighbors in \(D^*\) between
successive vertices, and put a
red neighbor in \(D^*\) at each endpoint. There are \(p-q\)
internal representatives. Each internal candidate set has size at
least \(d-2\mu\ge p-q\), so they can be chosen distinctly. After
they are removed, every endpoint candidate set has at least
\((d-\mu)-(p-q)\ge2q\) available vertices. Thus the \(2q\) endpoint representatives can
also be chosen distinctly. The resulting \(q\) alternating paths
use \(p+q\) distinct vertices of \(D^*\).

For each path, add \(a_0-1\) unused vertices of \(D^*\).  Its two
endpoints together with these vertices form the \((a_0+1)\)-class in
Lemma~\ref{lem:endpoints}, while \(Y_0\) is the \(a_0\)-class.
The resulting prescribed-endpoint path through \(Y_0\) closes the
original path into a cycle.  The \(q\) cycles use \(p+qa_0\) vertices
of \(D^*\).  Exactly \((h-q)a_0\) remain, and
Lemma~\ref{lem:nearcomplete} covers them in \(h-q\) balanced cycles.
\end{proof}

\begin{lemma}\label{lem:blockcover}
Let \(H[A,B]\) be bipartite, where \(|B|=b\ge3\), and suppose every
vertex of \(B\) has at most one non-neighbor in \(A\).  Put
\(E=\{a\in A:d(a,B)\le1\}\).  Then \(|E|\le1\).  If
\(|A\setminus E|\ge b\), then \(A\cup B\) is covered by at most
\(|E|+\ceil{(|A|-|E|)/b}\) red cycles.
\end{lemma}

\begin{proof}
The nonedge count gives \(|E|\le1\).  Partition \(A\setminus E\) into
blocks of size at most \(b\); because \(|A\setminus E|\ge b\), the
last block can be padded to size \(b\) with vertices of an earlier
block.  Lemma~\ref{lem:nearcomplete} gives one cycle for each block.
Cover the possible member of \(E\) by a singleton.
\end{proof}

\section{Minimum-counterexample structure}\label{sec:minimal}

The minimum-counterexample framework and the successor construction
below are cycle analogues of ideas developed for same-color path
covers by Pokrovskiy, Versteegen, and Williams~\cite{pokrovskiy2026proof}
and Chen and Chen~\cite[Section~3]{chen2026allorders}.  The use of
opposite-color cycle intersections and the cycle-closing arguments are
specific to the present setting.

Assume the theorem false and choose a counterexample of minimum order
\(n\).  The cases \(n\le3\) are immediate.  When \(n=4\), one color
contains either two independent edges or a triangle, so at most two
cycles of that color suffice.  When \(n=5\), the two color graphs
cannot both be forests; a monochromatic cycle and at most two
singletons suffice.  We also use the theorem of Faudree, Lesniak, and
Schiermeyer~\cite{faudree2009circumference} that every red--blue \(K_n\), \(n\ge6\), has a
monochromatic cycle of order at least \(\ceil{2n/3}\).  For
\(6\le n\le8\), such a cycle and at most two singletons give the
desired cover.  Hence \(n\ge9\).

Write
\begin{equation}\label{eq:nkr}
 k=\ceil{\sqrt n},\qquad n=(k-1)^2+r,\qquad1\le r\le2k-1.
\end{equation}
Choose a longest monochromatic cycle and call its color blue.  Denote
it by \(C\), orient it, and put \(X=V(C)\),
\(Y=V(K_n)\setminus X\), and \(w=|Y|\).
Observe that if \(w\le k-1\), then \(C\) and the \(w\) outside singletons give a blue cover with at most \(k\) cycles, this, together with the long-cycle theorem, gives
\begin{equation}\label{eq:long}
 |X|\ge\ceil{2n/3},\qquad k\leq w\le\floor{n/3}.
\end{equation}
For a vertex \(z\in V(K_n)\), let \(N^b(z)\) and \(N^r(z)\) denote its blue and red neighborhoods, respectively. For \(A\subseteq V(K_n)\), put
\[
N^b(z,A)=N^b(z)\cap A,\qquad
N^r(z,A)=N^r(z)\cap A,
\]
and
\[
d^b(z,A)=|N^b(z,A)|,\qquad
d^r(z,A)=|N^r(z,A)|.
\]
Let \(Y_0=\{z\in Y:\db(z,X)\le1\}\), \(Y_1=Y\setminus Y_0\), and
\(a_i=|Y_i|\). Lemma~\ref{lem:blockcover} applied to the red graph between $X$ and $Y$ gives us $a_1\geq1$.

The next lemma is the cycle counterpart of
\cite[Lemma~3.3]{pokrovskiy2026proof}; see also
\cite[Lemma~2.5]{chen2026allorders}.

\begin{lemma}\label{lem:easyblue}
There is a blue cover with at most \(1+a_0+\ceil{a_1/2}\) cycles.
Consequently
\begin{equation}\label{eq:basic}
 a_0+\ceil{a_1/2}\ge k.
\end{equation}
\end{lemma}

\begin{proof}
We first show that any two vertices \(z,z'\in Y_1\) lie on a common
blue cycle.  Let \(H\) be the blue graph consisting of \(C\), the
vertices \(z,z'\), and all blue edges from \(z\) and \(z'\) to \(C\).
Since \(z,z'\in Y_1\), each of them has at least two blue neighbors
on \(C\).

We claim that \(H\) is \(2\)-connected.  If one of \(z,z'\) is
deleted, the remaining outside vertex is still attached to the cycle
\(C\), so the graph remains connected.  If a vertex \(v\in V(C)\)
is deleted, then \(C-v\) is a path, and each of \(z,z'\) still has at
least one blue neighbor on this path.  Hence \(H-v\) is connected as
well.  Thus \(H\) is \(2\)-connected.

It is standard that any two vertices of a \(2\)-connected graph lie
on a common cycle.  Hence \(z\) and \(z'\) lie on a common blue cycle.

Now pair the vertices of \(Y_1\).  Each pair can be covered by one
blue cycle; if \(a_1\) is odd, cover the remaining vertex by a
singleton cycle.  Cover every vertex of \(Y_0\) by a singleton cycle
and include \(C\).  This gives a blue cover with at most
\[
 1+a_0+\ceil{a_1/2}
\]
cycles.

Since the coloring is a counterexample, this number must be at least
\(k+1\).  Therefore
\[
 a_0+\ceil{a_1/2}\ge k.
\]
\end{proof}

The following intersection bound is a cycle analogue of the
minimality argument used by Pokrovskiy, Versteegen, and Williams
in~\cite[Lemma~3.1]{pokrovskiy2026proof}.

\begin{lemma}\label{lem:intersection}
If \(R\) is a red cycle and \(B\) a blue cycle, then
\(|V(R)\cap V(B)|\le r-1\).
\end{lemma}

\begin{proof}
If \(S\subseteq V(R)\cap V(B)\) has size \(r\), then
\(n-r=(k-1)^2\).  By minimality, \(K_n-S\) has a same-color cover
with at most \(k-1\) cycles.  Add \(R\) if that cover is red and add
\(B\) if it is blue.
\end{proof}

For \(z\in Y\), let $U_z=$$\Nb(z,X)$ and let
\(S_z\) be the set of successors on \(C\).

\begin{lemma}\label{lem:successor}
For every \(z\in Y\):
\begin{enumerate}[label=(\roman*)]
\item \(U_z\) contains no two consecutive vertices of \(C\);
\item \(S_z\) is a red clique and \(z\) is red-complete to \(S_z\);
\item \(|S_z|\le r-1\);
\item every \(z'\in Y\setminus\{z\}\) has at most one blue neighbor
in \(S_z\).
\end{enumerate}
Moreover \(U_z\cap S_z=\varnothing\).
\end{lemma}

\begin{proof}
Write \(C=v_1v_2\cdots v_\ell v_1\), with indices modulo \(\ell\).
If \(v_i,v_{i+1}\in U_z\), replacing \(v_iv_{i+1}\) by
\(v_i z v_{i+1}\) extends \(C\).  This proves (i) and
\(U_z\cap S_z=\varnothing\).

Take distinct \(v_i,v_j\in U_z\).  Removing
\(v_iv_{i+1}\) and \(v_jv_{j+1}\), and adding \(v_i z v_j\) and
\(v_{i+1}v_{j+1}\), gives a blue cycle through \(V(C)\cup\{z\}\) if
\(v_{i+1}v_{j+1}\) is blue.  Thus \(S_z\) is a red clique.
Also \(zv_{i+1}\) is red by (i).  If \(z'\ne z\) were blue-adjacent
to two vertices \(v_{i+1},v_{j+1}\in S_z\), use instead
\(v_i z v_j\) and \(v_{i+1}z'v_{j+1}\) to extend \(C\) through two
outside vertices.  This proves (iv).
Finally, if \(|S_z|\ge r\), the red clique \(S_z\cup\{z\}\) contains
a red cycle meeting \(C\) in at least \(r\) vertices, contrary to
Lemma~\ref{lem:intersection}.
\end{proof}

Choose \(y\in Y_1\) with
\(\mu=\db(y,X)=\max_{z\in Y_1}\db(z,X)\), and put \(U=U_y\),
\(S=S_y\).  Thus \(|U|=|S|=\mu\),
\(U\cap S=\varnothing\), and \(2\le\mu\le r-1\).

\begin{lemma}\label{lem:cyclic}
If \(w-1\ge r-\mu\) and \(|X|\ge r+2\mu\), then the coloring is not
a minimum counterexample.
\end{lemma}

\begin{proof}
Put \(a=r-\mu\), choose distinct
\(z_1,\ldots,z_a\in Y\setminus\{y\}\), set \(z_0=y\), and let
\(W=X\setminus S\).  For every \(i\),
\[
 |\Nr(z_{i-1})\cap\Nr(z_i)\cap W|
 \ge |W|-2\mu=|X|-3\mu\ge a.
\]
Choose distinct representatives \(x_i\) and obtain the red path
\(yx_1z_1x_2z_2\cdots x_az_a\).
By Lemma~\ref{lem:successor}(iv), \(z_a\) has a red neighbor in
\(S\).  A Hamilton path through the red clique \(S\), with its other
end joined to \(y\), closes a red cycle meeting \(C\) in
\(a+\mu=r\) vertices, contrary to
Lemma~\ref{lem:intersection}.
\end{proof}

\begin{lemma}\label{lem:lowdefect}
If \(a_0\ge r-\mu\) and \(|X|\ge r+\mu+1\), then the coloring is not
a minimum counterexample.
\end{lemma}

\begin{proof}
Put \(a=r-\mu\), choose distinct \(z_1,\ldots,z_a\in Y_0\), set
\(z_0=y\), and let \(W=X\setminus S\).  The first common red
neighborhood has size at least \(|W|-\mu-1=|X|-2\mu-1\ge a\), and
every later one has size at least \(|W|-2=|X|-\mu-2\ge a\).
Choose distinct representatives and close through \(S\), exactly as
in Lemma~\ref{lem:cyclic}.
\end{proof}

\begin{lemma}\label{lem:concentration}
Assume \(\mu\ge3\).  For every \(z\in Y_1\),
\(|S_z\setminus S|\le r-1-\mu\) and
\(| U_z\setminus U|\le r-1-\mu\).
\end{lemma}

\begin{proof}
Suppose \(|S\cup S_z|\ge r\).  If \(r-\mu\ge2\), take
\(B\subseteq S_z\setminus S\) of size \(r-\mu\).  Since \(z\) has at
most one blue neighbor in \(S\), and \(y\) has at most one blue
neighbor in \(B\), the two red cliques can be oriented to form the
cycle \(y\,S\,z\,B\,y\).
It meets \(C\) in \(r\) vertices.

Now suppose \(r-\mu=1\).  If \(S_z\setminus S\) contains
two vertices \(b_1,b_2\), use \(y\,S\,z\,b_1b_2\,y\), reversing
\(b_1,b_2\) if needed.  If
\(S_z\setminus S=\{b\}\), choose \(i\in S_z\cap S\).
Because \(\mu\ge3\), a Hamilton path through \(S\setminus\{i\}\)
can end at a red neighbor of \(z\), and
\(y\,(S\setminus\{i\})\,z\,b\,i\,y\) again meets \(C\) in \(r\)
vertices.  These contradictions show
\(|S\cup S_z|\le r-1\).  The successor map is a bijection on \(C\),
so the two displayed differences have the same size.
\end{proof}

\begin{lemma}\label{lem:ordered}
Assume \(\mu\ge3\).  Choose \(2\le s\le\mu\) and
\(B\subseteq S\) with \(|B|=s\).  Put \(a=r-s\),
\(W=X\setminus(U\cup B)\), and \(\lambda=r-1-\mu\).
Let \(z_1,\ldots,z_a\) be distinct vertices of
\(Y\setminus\{y\}\), put \(\varepsilon_i=1\) for \(z_i\in Y_0\) and
\(\varepsilon_i=\lambda\) for \(z_i\in Y_1\), and define
\[
 L_i=|W|-\varepsilon_i-\varepsilon_{i+1}\ (i<a),\qquad
 L_a=|W|-\varepsilon_a.
\]
If the increasing rearrangement satisfies \(L_{j}\ge j\) for
\(1\le j\le a\), then the coloring is not a minimum counterexample.
\end{lemma}

\begin{proof}
Lemma~\ref{lem:concentration} gives the displayed lower bounds for the
required common red neighborhoods in \(W\).  Ordering these candidate
sets by size and choosing greedily gives distinct representatives
\(x_1,\ldots,x_a\).  Since \(|B|\ge2\), a Hamilton path in the red
clique \(B\) can end at a red neighbor of \(z_1\).  Therefore
\(y\,B\,z_1x_1z_2x_2\cdots z_ax_a\,y\) is a red cycle meeting \(C\)
in \(s+a=r\) vertices.
\end{proof}

\section{Proof of the main theorem}\label{sec:proof}

\subsection{The complement of the longest cycle}

\begin{proposition}\label{prop:w}
Every minimum counterexample satisfies \(w\le r-1\).
\end{proposition}

\begin{proof}
Suppose \(w\ge r\).  By \eqref{eq:long}, \(|X|\ge2w\).
If \(\mu=2\), the first condition of Lemma~\ref{lem:cyclic} holds.
If its second condition fails, then
\(2w\le|X|\le r+3\).  Since \(\mu\le r-1\) and \(w\ge r\), this
forces \(r=w=3\) and \(|X|=6\), so \(n=9\); but
\eqref{eq:nkr} gives \(r=5\), a contradiction.  Hence \(\mu\ge3\).

Put \(A=X\setminus U\) and \(\lambda=\max\{1,r-1-\mu\}\).
If Lemma~\ref{lem:cyclic} applies, we are done.  Its first condition
is automatic, so assume
\begin{equation}\label{eq:Xupper}
 |X|\le r+2\mu-1.
\end{equation}
If \(\mu\le r-2\), then
\(\lambda=r-1-\mu\), and \eqref{eq:long} and
\eqref{eq:Xupper} imply \(2n/3\le3r-3-2\lambda\).  Consequently
\[
 6\lambda\le7r-2(k-1)^2-9
 \le-2k^2+18k-18\le22,
\]
so \(\lambda\le3\).  If \(\mu=r-1\), then \(\lambda=1\).
Moreover \(w\ge2\lambda+2\): this is immediate in the latter case,
while in the former
\[
 2r\le2w\le|X|\le3r-3-2\lambda
\]
gives \(2\lambda\le r-3\).

It remains to check the second condition of
Lemma~\ref{lem:defectselection}.  If \(\mu=r-1\), then
\(|A|-w=(n-3w)+(w-r+1)\ge1\); otherwise,
\(|A|-w=(n-3w)+(w-r+1)+\lambda\ge\lambda+1\).  For
\(1\le\lambda\le3\) and \(w\ge2\lambda+2\), we have
\(\floor{\lambda(w-1)/(w-\lambda)}\le\lambda+1\).
Lemma~\ref{lem:defectselection} therefore gives a red cycle through
all of \(Y\) and \(w\) vertices of \(X\). Because \(w\ge r\), this
contradicts Lemma~\ref{lem:intersection}.
\end{proof}

\subsection{A spanning cycle and exact completion}

Call a red cycle \(Y\)-spanning if it contains every vertex of \(Y\)
and exactly \(w\) vertices of \(X\).

\begin{lemma}\label{lem:Yspan}
A \(Y\)-spanning red cycle exists in each of the following cases:
\begin{enumerate}[label=(\alph*)]
\item \(n-2w\ge2\mu\);
\item \(\mu\ge3\), and, for
\(A=X\setminus U\), \(\lambda=\max\{1,r-1-\mu\}\),
\[
 w\ge2\lambda+2,\qquad
 |A|-w\ge\floor{\frac{\lambda(w-1)}{w-\lambda}};
\]
\item \(\mu>w\).
\end{enumerate}
\end{lemma}

\begin{proof}
In (a), every vertex of \(Y\) has red degree at least
\(|X|-\mu\ge(|X|+w)/2\), so
Lemma~\ref{lem:smallclass} applies.  Part (b) is
Lemma~\ref{lem:defectselection}, using
Lemma~\ref{lem:concentration}.

For (c), consider the red bipartite graph between \(S\) and \(Y\).
The vertex \(y\) is complete to \(S\), and every other member of
\(Y\) misses at most one vertex of \(S\).  The total number of
nonedges is at most \(w-1\), so at most one vertex of \(S\) has red
degree at most one into \(Y\).  Since \(|S|=\mu>w\), choose a
\(w\)-set \(S'\subseteq S\) avoiding that vertex.
Lemma~\ref{lem:nearcomplete} applied to \(S'\cup Y\) gives a
Hamilton cycle.
\end{proof}

\begin{lemma}\label{lem:completion}
Assume \(a_0\ge3\) and let \(P\) be a \(Y\)-spanning red cycle.  Put
\(d=n-2w-1\), \(h=k-2\), and \(p=d-ha_0\); if \(p>0\), put
\(q=\min\{p,h\}\).  Assume \(d\ge a_0\).
If \(p\le0\), or if
\begin{equation}\label{eq:packconditions}
 p\le qa_1,\qquad d-2\mu\ge p-q,\qquad d-\mu\ge p+q,
\end{equation}
then \(K_n\) has a red cover with at most \(k\) cycles.
\end{lemma}

\begin{proof}
Among the \(X\)-vertices outside \(P\), remove the unique possible
vertex having red degree at most one into \(Y_0\); if there is no such
vertex, remove any one vertex.  Cover it by a singleton.
Lemma~\ref{lem:packing} covers the remaining \(d\) vertices in at
most \(h\) cycles.  Together with \(P\), the total is \(1+1+h=k\).
\end{proof}

\subsection{The uniform range \texorpdfstring{\(k\ge5\)}{k >= 5}}

\begin{proposition}\label{prop:a0large}
No minimum counterexample with \(k\ge5\) has \(a_0\ge3\).
\end{proposition}

\begin{proof}
Assume neither Lemma~\ref{lem:cyclic} nor
Lemma~\ref{lem:lowdefect} applies.

First let \(k\ge6\).  Proposition~\ref{prop:w} gives
\(|X|\ge(k-1)^2+1\ge2r\ge r+\mu+1\).
The failure of Lemma~\ref{lem:lowdefect} therefore gives
\begin{equation}\label{eq:mubound}
 \mu\le r-a_0-1.
\end{equation}
We claim
\begin{equation}\label{eq:gap}
 n-2w\ge2\mu+1.
\end{equation}
Otherwise, using \(w\le r-1\), \(a_0\ge3\), and
\eqref{eq:mubound}, we obtain
\(n\le2w+2\mu\le2(r-1)+2(r-4)=4r-10\).
Thus \((k-1)^2+10\le3r\le6k-3\), which is impossible because
\(k^2-8k+14>0\) for \(k\ge6\).

Lemma~\ref{lem:Yspan}(a) gives a \(Y\)-spanning red cycle.
Now \(d=n-2w-1\ge2\mu\), and
\(d\ge w-1\ge a_0\).  We verify
\eqref{eq:packconditions}.  If \(p\le h\), then \(q=p\); the first
two inequalities are immediate, and
\[
 d-\mu-2p=2ha_0-d-\mu
 \ge h(a_0-1)-\mu\ge1.
\]
For the last inequality use \eqref{eq:mubound},
\(a_0\ge3\), and \(r\le2k-1\).
If \(p>h\), then \(q=h\).  Since \(w\ge k\), we have
\(n\le kw\), whence \(d\le hw\) and \(p\le ha_1\).  The remaining
two conditions follow from
\[
 h(a_0+1)-2\mu
 \ge ka_0+k-2r\ge2,\qquad
 h(a_0-1)-\mu\ge1.
\]
Lemma~\ref{lem:completion} gives a contradiction.

It remains to check \(k=5\), where \(r\le9\) and \(|X|\ge17\).
If \(|X|<r+\mu+1\), then \(n=25\), \(r=9\), \(w=8\), and \(\mu=8\).
For \(A=X\setminus U\) and \(\lambda=1\), we have
\(|A|=9\) and \(|A|-w=1\), so
Lemma~\ref{lem:Yspan}(b) applies.  Here \(d=8\) and
\(p=8-3a_0<0\).

Suppose \(|X|\ge r+\mu+1\), so \eqref{eq:mubound} holds.
If \eqref{eq:gap} fails, the only possibility is \(n=25\), \(r=9\),
\(w=8\), \(a_0=3\), and \(\mu=5\).
Here \(\lambda=3\), \(|A|=12\), and
\[
 |A|-w=4=\floor{\frac{3\cdot7}{8-3}},
\]
so Lemma~\ref{lem:Yspan}(b) applies; again \(d=8\) and \(p=-1\).
Outside these two boundary configurations,
\eqref{eq:gap} holds and the preceding packing calculation remains
valid with \(h=3\).
\end{proof}

\begin{proposition}\label{prop:a0small}
No minimum counterexample with \(k\ge5\) has \(a_0\le2\).
\end{proposition}

\begin{proof}
Suppose, for a contradiction, that there is a minimum counterexample with $k\geq 5$ and $a_0\leq 2$. If \(a_0+\ceil{a_1/2}\ge k+1\), then
\(a_0\ge2k+1-w\ge3\) by Proposition~\ref{prop:w}, a contradiction.
Thus \(a_0+\ceil{a_1/2}=k\).

The case \(a_0=0\) would require \(w\ge2k-1\), again contradicting
\(w\le r-1\le2k-2\).

\smallskip
\noindent\emph{Case \(a_0=1\).}
The equality \(a_0+\ceil{a_1/2}=k\) and Proposition~\ref{prop:w} force
\begin{equation}\label{eq:a01}
 n=k^2,\quad r=2k-1,\quad w=2k-2,\quad
 a_1=2k-3,\quad |X|=k^2-2k+2.
\end{equation}
If \(\mu\le k-1\), Lemma~\ref{lem:cyclic} applies, because
\[
 w-1=2k-3\ge r-\mu,\qquad
 |X|-(r+2\mu)\ge(k-1)(k-5).
\]
Suppose \(\mu\ge k\), put \(a=r-\mu\). We want to apply
Lemma~\ref{lem:ordered} with \(B=S\) and $W=X\setminus(U\cup S)$. Then

\[
 1\le a\le k-1,\qquad \lambda=a-1,\qquad
 |W|=|X|-2\mu.
\]

For \(a=1\), choose \(z_1\in Y_1\). Since then
\(\lambda=a-1=0\), we have \(\varepsilon_1=0\), and hence
$L_1=|W|.$ 

For \(a=2\), choose
\(z_1\in Y_1\) and \(z_2\in Y_0\). Now \(\lambda=1\), so
$\varepsilon_1=\varepsilon_2=1.$
Thus the two lower bounds in Lemma~\ref{lem:ordered} are
$L_1=|W|-2$, $L_2=|W|-1.$
In both cases the required distinct representatives can therefore be
chosen directly. 

Suppose now that \(a\ge3\).
Place $z_1,\dots,z_a$, with $z_{a-1}\in Y_0$ and the rest of the vertices in $Y_1$. 

 Let $L_i$ be defined as in Lemma~\ref{lem:ordered}. For $i\leq a-3$, $L_i=
 |W|-2\lambda.$
Since $k\geq5$, using the definitions of \(W\) and \(\lambda\), we obtain
\[
 \begin{aligned}
 |W|-2\lambda
 &=|X|-2\mu-2(r-1-\mu)\\
 &=(k^2-2k+2)-2(2k-1)+2\\
 &=k^2-6k+6 \geq a-3.
 \end{aligned}
\]
 $L_{a-2}=L_{a-1}=|W|-1-\lambda.$
Thus
\[
 \begin{aligned}
 |W|-1-\lambda
 &=|X|-2\mu-1-(r-1-\mu)\\
 &=(k^2-2k+2)-2(2k-1)+a\\
 &=k^2-6k+4+a \geq a-1.
 \end{aligned}
\]

Finally, $L_a=|W|-\lambda.$
Again,
\[
 \begin{aligned}
 |W|-\lambda
 &=|X|-2\mu-(r-1-\mu)\\
 &=k^2-6k+5+a \geq a.
 \end{aligned}
\]
By Lemma~\ref{lem:ordered}, this is not a minimum counterexample.

\smallskip
\noindent\emph{Case \(a_0=2\).}
Now
\begin{equation}\label{eq:a02}
 w\in\{2k-3,2k-2\},\qquad n\in\{k^2-1,k^2\}.
\end{equation}
For \(\mu\le k-1\), both conditions of
Lemma~\ref{lem:cyclic} hold, except possibly when
\begin{equation}\label{eq:exception}
 n=k^2,\quad r=2k-1,\quad w=2k-3,\quad\mu=2.
\end{equation}
We treat this exceptional configuration separately at the end of the proof.

Let \(\mu\ge k\), set \(a=r-\mu\). We will apply
Lemma~\ref{lem:ordered} with \(B=S\).  

If \(a\ge5\), place $z_1,\dots,z_a$ with $z_{a-3},z_{a-1}\in Y_0$ and the rest in $Y_1$. We define $L_i$ as in Lemma~\ref{lem:ordered}. For $i\leq a-5$,
\[
L_i = |W|-2\lambda
=|X|-2\mu-2(r-1-\mu)
=|X|-2r+2\geq a-5.
\]
For $i\in\{a-4,a-3,a-2,a-1\}$,  
\[L_i=|W|-\lambda-1
=|X|-2\mu-(r-1-\mu)-1
=|X|-r-\mu \geq a-1.
\]
Finally, 
\[
 L_a= |W|-\lambda
=|X|-r-\mu+1\geq a,
\]
and Lemma~\ref{lem:ordered} applies.

It remains to consider \(a\le4\).  We choose the vertices
\(z_1,\ldots,z_a\) explicitly according to whether they lie in
\(Y_0\) or \(Y_1\).

If \(a=4\), choose
\(z_1,z_3\in Y_1\), \(z_2,z_4\in Y_0\).
Here \(\lambda=a-1=3\).  For \(i=1,\ 2,\ 3\), the pair
\(z_i,z_{i+1}\) consists of one vertex of \(Y_0\) and one vertex of
\(Y_1\), so 
\(
 L_i=|W|-\lambda-1=|W|-4.
\)
Since \(z_4\in Y_0\), $L_4 = |W|-1.$
Over the possibilities in \((10)\), we have \(|W|\ge7\).
Hence $L_1$, $L_2$, $L_3\geq3$ and $L_4\geq6$.

If \(a=3\), choose
\(z_1,z_3\in Y_0\), \(z_2\in Y_1.
\)
Now \(\lambda=a-1=2\). $L_1=L_2= |W|-\lambda-1=|W|-3$,
while $L_3=|W|-1.$
In this case \(|W|\ge5\), so $L_1=L_2\geq 2$, $L_3\geq 4$.

If \(a=2\), choose both \(z_1,z_2\) in \(Y_0\).  Since
\(\lambda=1\), in this case $|W|\geq 3$.
\(
 L_1=|W|-2\geq 1,
\)
\(
 L_2=|W|-1\geq 2.
\)

Finally, if \(a=1\), choose \(z_1\in Y_1\).  Here
\(\lambda=0\), so 
\(L_1=|W|,
\)
which is at least \(1\). 

In every case above, Lemma~\ref{lem:ordered} applies.

Finally consider \eqref{eq:exception}.
Lemma~\ref{lem:Yspan}(a) gives a \(Y\)-spanning cycle.  Remove two
remaining \(X\)-vertices, including the at most two vertices of red
degree at most one into \(Y_0\), and cover them by singletons.  For
the retained vertices,
\[
 d=(k-2)^2,\quad h=k-3,\quad
 p=k^2-6k+10,\quad q=h,\quad a_1=2k-5.
\]
The three conditions of Lemma~\ref{lem:packing} are equivalent to
\(k^2-5k+5>0\), \(3k-13\ge0\), and \(k-5\ge0\).
Hence a red cover with \(k\) cycles exists.
\end{proof}

\begin{corollary}\label{cor:klarge}
No minimum counterexample has \(k\ge5\).
\end{corollary}

\begin{proof}[Proof of Theorem~\ref{thm:main}]
Suppose that a counterexample exists and choose one of minimum order.
The cases \(n\le8\) were settled at the beginning of
Section~\ref{sec:minimal}.  Proposition~\ref{prop:w} gives
\(w\le r-1\), and Propositions~\ref{prop:a0large} and
\ref{prop:a0small} eliminate every possibility with \(k\ge5\).
Appendix~\ref{app:k4}, beginning with
Proposition~\ref{prop:k4reduce}, eliminates all cases with \(k=4\),
while Lemma~\ref{lem:k3} in Appendix~\ref{app:k3} eliminates \(k=3\).
Thus no minimum counterexample exists.
\end{proof}

\section*{Acknowledgements}
J. Teodomiro was supported by CNPq, grant 140465/2026-0, and Coordena\c{c}\~{a}o de Aperfeiçoamento de Pessoal de N\'{i}vel Superior - Brasil (CAPES) - Finance Code 001.

\appendix

\section{The structural core \texorpdfstring{\(k=4\)}{k = 4}}\label{app:k4}

\begin{proposition}\label{prop:k4reduce}
If a minimum counterexample has \(k=4\), then
\(n\in\{14,15,16\}\), \(w\in\{4,5\}\), and
\[
 (w,a_0,a_1)\in\{(4,3,1),(5,4,1),(5,3,2),(5,2,3)\}.
\]
\end{proposition}

\begin{proof}
By Proposition~\ref{prop:w}, and
\eqref{eq:long},
\[
 4\le w\le n/3,\qquad w\le r-1=n-10.
\]
Thus \(14\le n\le16\) and \(w\in\{4,5\}\).
The four triples follow immediately from
\(a_1\ge1\) and \(a_0+\ceil{a_1/2}\ge4\), as in Lemma~\ref{lem:easyblue}.
\end{proof}

In the following displayed cycles, a symbol such as \(S\) or \(P_i\)
inside a vertex sequence denotes a Hamilton path through the indicated
red clique or red path, oriented so that the displayed attachment
edges are red.

\begin{lemma}\label{lem:k4a11}
No minimum counterexample with \(k=4\) has \(a_1=1\).
\end{lemma}

\begin{proof}
Write \(Y_1=\{y\}\), so \(a_0=w-1\).  The red clique
\(S\cup\{y\}\) contains a cycle \(Q\).  Put \(\theta=n-3w\ge0\) and
\(A=X\setminus S\).
If \(\mu\ge\theta+1\), then
\(|A|\le2w-1=2a_0+1\).  Also
\(|A\setminus E|\ge a_0\): indeed \(|X|\ge10\),
\(\mu\le|X|/2\), and \(a_0\le4\).
Lemma~\ref{lem:blockcover} covers \(A\cup Y_0\) with at most three
cycles, and \(Q\) is the fourth.

Suppose \(\mu\le\theta\).  Since \(\theta\le w\), we have
\(n-2w=\theta+w\ge2\mu\), so Lemma~\ref{lem:Yspan}(a) applies.  In
Lemma~\ref{lem:completion}, \(h=2\) and
\(p=n-4w+1=\theta-w+1\le1\).
If \(p\le0\), we are done.  If \(p=1\), then
\(n=16,w=4,\theta=4\).  The value \(\mu=4\) is excluded by
Lemma~\ref{lem:lowdefect}, so \(\mu\le3\), and all three conditions
in \eqref{eq:packconditions} hold.
\end{proof}

We shall use the elementary fact that a graph of matching number at
most one is a star together with isolated vertices, or a triangle
together with isolated vertices.

\begin{lemma}\label{lem:k4a12}
No minimum counterexample with \(k=4\) has \(a_1=2\).
\end{lemma}

\begin{proof}
Here \(w=5,a_0=3,n\in\{15,16\}\).  Write \(Y_1=\{y,z\}\) and
\(Y_0=\{u_1,u_2,u_3\}\), where \(y\) has blue degree \(\mu\) into
\(X\).
If \(n-2w\ge2\mu\), Lemmas~\ref{lem:Yspan}(a) and
\ref{lem:completion} apply, since
\(d=n-11\) and \(p=d-6<0\).
Thus \(\mu\ge3\) for \(n=15\), and \(\mu\ge4\) for \(n=16\).
By Lemma~\ref{lem:successor}, $z$ has at least $2$ red neighbors in $S$, say $x_1$ and $x_2$. Let $P_1$ and $P_2$ be a partition of $S$ into two red paths, where $x_1$ and $x_2$ are in the extremities of $P_1$ and $P_2$, respectively. Then $yP_1zP_2y$ is a red cycle covering all of $S$. Let $A=X\setminus S$, then, in both cases, $|A|\leq7$ and Lemma~\ref{lem:blockcover} covers $A\cup Y_0$ with at most three further red cycles.
\end{proof}

It remains to handle \(a_1=3\).  Write
\[
 Y_0=\{u,v\},\qquad Y_1=Z=\{y,z_2,z_3\}.
\]

\begin{lemma}\label{lem:Zstructure}
The set \(Z\) is a red triangle, and the blue graph on \(Y\) has
matching number at most one.
\end{lemma}

\begin{proof}
Adjoin \(Z\) and its blue edges to \(C\).  Every vertex of \(Z\) has
at least two neighbors on \(C\), so deletion of any one vertex leaves
the graph connected; hence it is \(2\)-connected.  In a
\(2\)-connected graph, a prescribed edge and a third vertex lie on a
common cycle (apply the two-fan lemma from the third vertex to the
ends of the edge).  Thus a blue edge in \(Z\) would lie on a blue
cycle through all of \(Z\).  Together with \(C\) and the singleton
cycles \(u,v\), this gives four blue cycles.  Hence \(Z\) is a red
triangle.  Two independent blue edges in \(Y\), the remaining
singleton, and \(C\) give the same contradiction.
\end{proof}

\begin{lemma}\label{lem:compensation}
If \(q\in Y_0\) is blue-adjacent to a vertex of \(Z\), then \(q\) is
red-complete to \(X\).
\end{lemma}

\begin{proof}
Suppose \(qz\) and \(qx\), with \(z\in Z,x\in X\), are blue.
Choose \(x'\in\Nb(z,X)\setminus\{x\}\).  The edges
\(qz,zx',qx\), and an \(x\)--\(x'\) arc of \(C\) form a blue cycle
through \(q,z\).  The other two vertices of \(Z\) lie on a common
blue cycle by Lemma~\ref{lem:easyblue}.  These two cycles, \(C\), and
the remaining \(Y_0\)-singleton form a four-cycle blue cover.
\end{proof}

\begin{lemma}\label{lem:k4muge3}
No minimum counterexample with \(k=4,a_1=3,\mu\ge3\) exists.
\end{lemma}

\begin{proof}
If \(n=15,\mu=3\), apply Lemma~\ref{lem:ordered} with \(s=2\)
and $z_1,z_2\in Y_1$, $z_3,z_4\in Y_0$.  The increasing lower bounds are exactly
\(1,2,3,4\).

We give the remaining absorbers.  All representatives are chosen
outside the displayed successor-clique vertices.

If \(|X|=11,\mu=3\), use
\[
 ySux_1z_2x_2vx_3z_3x_4y.
\]
The first three candidate sets have size at least four and the last
has size at least two.  Choose the last representative first, then
the other three; the bounds \(2,4,4,4\) give distinct choices.

Let \(\mu=4\).  If \(|X|=10\), at least one edge between
\(\{u,v\}\) and \(\{z_2,z_3\}\) is red, since otherwise the blue
graph on \(Y\) has a matching of size two.  Relabel it as \(vz_3\)
and use
\[
 ySz_2x_1ux_2vz_3y.
\]
The two representative sets have sizes at least one and four.
If \(|X|=11\), use
\[
 ySz_2x_1ux_2vx_3z_3y;
\]
the three lower bounds \(2,5,2\), when sorted as \(2,2,5\), satisfy
Hall's condition.

Let \(\mu=5\) and \(|X|=10\).  If the red bipartite graph between
\(\{z_2,z_3\}\) and \(\{u,v\}\) has a perfect matching, say
\(z_2u,z_3v\), choose
\[
 x\in\Nr(u)\cap\Nr(v)\setminus S
\]
and use \(ySz_2uxvz_3y\).  Otherwise choose one blue edge blocking
each of the two red perfect matchings.  These two blue edges cannot
be independent by Lemma~\ref{lem:Zstructure}; hence they form a
star.  Up to relabeling, its center is \(z_2\) or \(u\).
In the first case Lemma~\ref{lem:compensation} makes \(u,v\)
red-complete to \(X\), while \(z_3u\) and \(vy\) are red; use
\[
 ySz_2z_3uxvy.
\]
In the second case \(u\) is red-complete to \(X\), both
\(vz_2,vz_3\) are red, and \(v\) has a red neighbor
\(x\in X\setminus S\); use
\[
 ySuxvz_2z_3y.
\]

Finally let \(\mu=5,|X|=11\), put \(W=X\setminus S\), and define
\[
 F_1=\Nr(z_2)\cap\Nr(u)\cap W,\qquad
 F_3=\Nr(v)\cap\Nr(z_3)\cap W.
\]
If \(F_1=F_3=\varnothing\), then each pair among \(y,z_2,z_3\)
has at least four common blue neighbors in the six-set \(W\).
We may therefore choose distinct representatives from the three corresponding common blue neighborhoods, giving
a blue \(6\)-cycle through \(Z\); with \(C,u,v\) this is a
four-cycle blue cover.  Relabel so that \(x_1\in F_1\).  Choose
\[
 x_2\in\Nr(u)\cap\Nr(v)\cap(W\setminus\{x_1\}),
\quad
 x_3\in S\cap\Nr(v)\cap\Nr(z_3).
\]
The first candidate set for \(x_2\) has size at least four before
deleting \(x_1\), and the set for \(x_3\) has size at least three.
Choose a four-set \(B\subseteq S\setminus\{x_3\}\).  Then
\[
 yBz_2x_1ux_2vx_3z_3y
\]
is red.  Every displayed red cycle in this proof contains exactly
\(r\) vertices of \(C\), contrary to
Lemma~\ref{lem:intersection}.
\end{proof}

\begin{lemma}\label{lem:k4mu2}
No minimum counterexample with \(k=4,a_1=3,\mu=2\) exists.
\end{lemma}

\begin{proof}
For \(n=15\), Lemma~\ref{lem:cyclic} applies, so \(n=16\).
For \(z_i\in Z\), let \(S_i\) be its red successor edge.

First suppose two distinct successor edges meet.  Their union is a
three-vertex red path \(P_X\) in \(X\), and it can be oriented so
that
\[
 z_1P_Xz_2
\]
is red.  In the eight-set \(X\setminus V(P_X)\), the four mixed
common red neighborhoods needed below each have size at least five.
Thus distinct representatives give
\[
 z_1P_Xz_2x_1ux_2z_3x_3vx_4z_1.
\]

If two successor edges are disjoint, orient one from \(z_1\) toward
\(u\), and the other from \(v\) toward \(z_2\).  In their
seven-vertex complement, the two mixed candidate sets have size at
least four, and the common red candidate set for \(z_2,z_1\) has
size at least three.  Hence
\[
 z_1S_1ux_1z_3x_2vS_2z_2x_3z_1
\]
is red.  Both cycles meet \(C\) in seven vertices.

The only remaining configuration is
\(S_1=S_2=S_3\): if two equal edges and a third different edge
occur, the third either meets them (the first case) or is disjoint
(the second case).  Equality of the successor sets implies equality
of the blue neighborhoods, say \(U\).  Thus \(Z\) is
red-complete to \(A=X\setminus U\), where \(|A|=9\).
Choose \(A'\subseteq A\) of size five.
Lemma~\ref{lem:nearcomplete} gives a red Hamilton cycle on
\(A'\cup Y\): vertices of \(A'\) have their three neighbors in
\(Z\), and each member of \(Y_0\) misses at most one vertex of
\(A'\).  Put \(A\setminus A'=\{a_1,a_2,a_3,a_4\}\).  Choose
\(q\in Y_0\) and order these four vertices so that
\(qa_1,qa_4\) are red.  Then
\[
 a_1z_1a_2z_2a_3z_3a_4qa_1
\]
is a second red cycle.  These two cycles and the two singleton
vertices of \(U\) form a four-cycle red cover.
\end{proof}

\section{The final case \texorpdfstring{\(k=3\)}{k = 3}}\label{app:k3}

\begin{lemma}\label{lem:k3}
No minimum counterexample has \(k=3\).
\end{lemma}

\begin{proof}
The earlier reductions force
\[
 n=9,\qquad w=3,\qquad a_0=2,\qquad a_1=1.
\]
Write \(Y_0=\{u,v\}\), \(Y_1=\{z\}\), \(S=S_z\), and
\(U=\Nb(z,X)\).  A blue edge in \(Y\), the remaining singleton,
and \(C\) would give three blue cycles, so \(Y\) is a red triangle.
Since \(U\) has no consecutive vertices on the six-cycle \(C\),
\(\mu=|U|\in\{2,3\}\).

If \(\mu=3\), the two vertices \(u,v\), each missing at most one
vertex of \(U\), have a common red neighbor \(c\in U\).
Orient a Hamilton path through \(S\) so that its last vertex is
red-adjacent to \(u\).  Then
\[
 zS\,u\,c\,v\,z
\]
is a red cycle covering \(Y\) and four vertices of \(X\); cover the
other two vertices by singletons.

If \(\mu=2\), put \(R=X\setminus(S\cup U)\), so \(|R|=2\).
Choose \(r_0\in R\cap\Nr(v)\).  The set
\(X\setminus(S\cup\{r_0\})\) has size three, while \(u,v\) each
miss at most one of its vertices; hence it contains
\[
 c\in\Nr(u)\cap\Nr(v).
\]
Orient \(S\) toward \(u\).  Since \(r_0\notin U\), the cycle
\[
 zS\,u\,c\,v\,r_0\,z
\]
is red and again leaves only two \(X\)-vertices for singleton
cycles.
\end{proof}

\bibliographystyle{abbrv}
\bibliography{references}

\end{document}